\documentclass[11pt]{article}

\usepackage{amsfonts}
\usepackage{amssymb}
\usepackage{amsmath}
\usepackage{authblk}
\usepackage{bbm}
\usepackage{hyperref}
\usepackage{graphicx}
\usepackage{float}
\usepackage{overpic}
\usepackage{geometry}
\usepackage{epstopdf}
\usepackage{multirow}
\usepackage{subcaption}
\usepackage[official]{eurosym}

\providecommand{\U}[1]{\protect\rule{.1in}{.1in}}

\newtheorem{theorem}{Theorem}[section]

\newtheorem{definition}[theorem]{Definition}

\newtheorem{lemma}[theorem]{Lemma}

\newtheorem{proposition}[theorem]{Proposition}
\newtheorem{remark}[theorem]{Remark}

\newenvironment{proof}[1][Proof]{\noindent\textbf{#1.} }{\ \rule{0.5em}{0.5em}}
\usepackage{amsmath,systeme}
\usepackage{lineno}
\usepackage[dvipsnames]{xcolor}
\usepackage{indentfirst}

\DeclareMathOperator*{\argmax}{arg\,max}

\begin{document}

\title{Optimal installation of renewable electricity sources: the case of Italy}

\author{Almendra Awerkin\thanks{{\tt awerkin@math.unipd.it}}}
\author{Tiziano Vargiolu\thanks{{\tt vargiolu@math.unipd.it}}}
\affil{Department of Mathematics "Tullio Levi-Civita"\\
  University of Padua}

\maketitle

\begin{abstract} 
Starting from the model in \cite{KV}, we test the real impact of current renewable installed power in the electricity price in Italy, and assess how much the renewable installation strategy which was put in place in Italy deviated from the optimal one obtained from the model in the period 2012--2018. To do so, we consider the Ornstein-Uhlenbeck (O-U) process, including an exogenous increasing process influencing the mean reverting term, which is interpreted as the current renewable installed power. 
We estimate the parameters of this model by using real data of electricity prices and energy production from photovoltaic and wind power plants from the six main Italian price zones. We obtain that the model fits well the North, Central North and Sardinia zones: among these zones, the North is impacted by photovoltaic production, Sardinia by wind and the Central North does not present significant price impact. 
Then we implement the solution of the singular optimal control problem of installing renewable power plants
, in order to maximize the profit of selling the produced energy in the market net of installation costs. We extend the results of \cite{KV} to the case when no impact on power price is presented, and to the case when $N$ players can produce electricity by installing renewable power plants.  To this extent, we analyze both the concepts of Pareto optima and of Nash equilibria. For this latter, we present a verification theorem in the 2-player case, and an explicit characterization of a Nash equilibrium in the no-impact case. 
We are thus able to describe the optimal strategy and compare it with the real installation strategy that was put in place in Italy. 
\end{abstract} 


\section{Introduction}

The paper \cite{KV} describes the irreversible installation problem of photovoltaic panels for an infinitely-lived profit maximizing power-producing company, willing to maximize the profits from selling electricity in the market. The power price model used in that paper assumes that the company is a large market player, so its installation has a negative impact on power price. More in detail, the power price is assumed to follow an additive mean-reverting process (so that power price could possibly be negative, as it happens in reality), where the long-term mean decreases as the cumulative installation increases. The resulting optimal strategy is to install the minimal capacity so that the power price is always lower than a given nonlinear function of the capacity, which is characterized by solving an ordinary differential equation deriving from a free-boundary problem. 
The aim of this paper is to validate empirically that model, extending it also to wind power plants' installation, by using time series of the main Italian zonal prices and power production, and to assess how much the renewable installation strategy in Italy deviated from the optimal one obtained by \cite{KV} in the period 2012--2018. In doing so, we also extend the theoretical results of \cite{KV} to the (easier) case when the amount of installed renewable capacity has no impact on power prices, and to the case when several power producers are present in the market: in doing so, we investigate and compare Pareto optima and Nash equilibria in this situation.

It is common in  literature to model electricity prices via a mean-reverting behaviour, and to include (jump) terms representing the seasonal fluctuations and daily spikes, cf. \cite{Borovkova,CarFig,German,Weron} among others. Here, in analogy with \cite{KV}, we do not represent the spikes and seasonal fluctuations with the following argument: the installation time of solar panels or wind turbines usually takes several days or weeks, which makes the power producers indifferent of daily or weekly spikes. Also, the high lifespan of renewable power plants and the underlying infinite time horizon setting allow us to neglect the seasonal patterns. We therefore assume that the electricity's fundamental price has solely a mean-reverting behavior, and evolves according to an Ornstein-Uhlenbeck (O-U) process\footnote{We allow for negative prices by modelling the electricity price via an Ornstein-Uhlenbeck process. Indeed, negative electricity prices can be observed in some markets, for example in Germany, cf. \cite{NYTimes}.}. We are also neglecting the stochastic and seasonal effects of renewable power production. In fact, photovoltaic production has obvious seasonal patterns (solar panels do not produce power during the night and produce less in winter than in summer), and both solar and wind power plants are subject to the randomness affecting weather conditions. However, since here we are interested to a long-term optimal behaviour, we interpret the average electricity produced in a generic unit of time as proportional to the installed power. All of this can be mathematically justified if we interpret our fundamental price to be, for example, a weekly average price as e.g. in \cite{BPPB,FonVarZor,GPP}, who used this representation exactly to get rid of daily and weekly seasonalities. 

In order to represent price impact of renewables in power prices, which is more and more observed in several national power markets, we follow the common stream in literature (also in analogy with \cite{KV}) and represent renewable capacity installation as a non-decreasing process, thus resulting in a singular control problem. This is also analogous to other papers modelling price impact: for example, in problems of optimal execution, \cite{Becherer2} and \cite{Becherer} take into account a multiplicative and transient price impact, whereas \cite{guo} considers an exponential parametrization in  a geometric Brownian motion setting allowing for a permanent price impact. Also, a price impact model has been studied by \cite{Zervos}, motivated by an irreversible capital accumulation problem with permanent price impact, and by \cite{KF}, in which the authors consider an extraction problem with Ornstein-Uhlenbeck dynamics and transient price impact. 
In all of the aforementioned papers on price impact models dealing with singular stochastic controls \cite{Zervos,Becherer2,Becherer,KF,guo}, the agents' actions can lead to an immediate jump in the underlying price process, whereas in our setting, it cannot. 
Our model is instead analogous to \cite{CFSV,CarJai}, which show how to incorporate a market impact due to cross-border trading in electricity markets, and to \cite{RVG}, which models the price impact of wind electricity production on power prices. In these latter models, price impact is localized on the drift of the power price. 

In order to validate our model, we use a dataset of weekly Italian prices, together with photovoltaic and wind power production, of the six main Italian price zones (North, Central North, Central South, South, Sicily and Sardinia), covering the period 2012--2018. In principle, both photovoltaic and wind power production could have an impact on power prices, so we start by estimating parameters of an ARX model where both photovoltaic and wind power production are present as exogenous variables: the parameters of this discrete time model will then be transformed in parameters for the continuous time O-U model by standard techniques, see e.g. \cite{brigo}. Unfortunately, for three price zones we find out that our O-U model, even after correcting for price impacts, produce non-independent residuals. This is an obvious indication that the O-U model is too simple for these zones, and one should instead use more sophisticated models, like CARMA ones (see e.g. \cite{Benth}): we leave this part for future research. For the remaining three zones, we find out that, for each zone, at most one of the two renewable sources has an impact: in particular, power price in the North is only impacted by photovoltaic production, and in Sardinia only by wind production, while in Central North is not impacted by any of them. Thus, we are able to model the optimal installation problem for North and Sardinia using the theory existing in \cite{KV}. Instead, for the installation problem in Central North, we must solve an instance of the problem with no price impact: this can  be derived as a particular case of the results in \cite{KV},  and  results in a much more elementary formulation than the general case in \cite{KV}. More in detail, we obtain that the function of the capacity which should be hit by the power price in order to make additional installation is in this case equal to a constant, obtained by solving a nonlinear equation. The corresponding optimal strategy should thus be to not install anything until the price threshold is hit, and then to install the maximum possible capacity. 

The second aim of our paper is to check the effective installation strategy, in the different price zones, against the optimal one obtained theoretically. In doing so, we must take into account the fact that the Italian market is liberalized since about two decades, thus there is not a single producer which can impact prices by him/herself, but rather prices are impacted by the cumulative installation of all the power producers in the market. We thus extend our model by formulating it for $N$ players who can install, in the different price zones, the corresponding impacting renewable power source, monotonically and independently of each other: the resulting power price will be impacted by the sum of all these installations, while each producer will be rewarded by a payoff corresponding to their installation. The resulting $N$-player nonzero-sum game can be solved with different approaches. 
A formulation requiring a Nash equilibrium would result in a system of $N$ variational inequalities with $N + 1$ variables (see e.g. \cite{DeAFer} and references therein), which would be quite difficult to treat analytically. 
We choose instead to seek for Pareto optima first. One easy way to achieve this is to assume, in analogy with \cite{ferrari}, the existence of a "social planner" which maximizes the sum of all the $N$ players' payoffs, under the constraint that the sum of their installed capacity cannot be greater than a given threshold (which obviously represents the physical finite capacity of a territory to support power plants of a given type). We prove that, in our framework, this produces Pareto optima. More in detail, by summing together all the $N$ players' installations in the social planner problem, one obtains the same problem of a single producer, which has a unique solution that represents the optimal cumulative installation of all the combined producers. Though with this approach it is not possible to distinguish the single optimal installations of each producer, we can assess how much the effective cumulative installation strategy which was carried out in Italy during the dataset's period differs from the optimal one which we obtained theoretically.  To give an idea of what we instead would get when searching for Nash equilibria, we present the case $N = 2$ and formulate a verification theorem that the value functions of each player should satisfy. Here we want to point out a difference which arises in our problem with respect to the current stream of literature. In fact, in stochastic singular games the usual framework is that a player can act only when the other ones are idle, see e.g. \cite{CGX,DeAFer,GTX,Xin}. Here instead we take explicitly into consideration the possibility that both players acts (i.e. install) simultaneously. 
This possibility will be confirmed in Section 5.3, where (in the case with no market impact) we present a Nash equilibrium where both players install simultaneously. Another peculiarity is that this equilibrium induces the players to install {\em before} than when they would have done under a Pareto optimum. This is the converse phenomenon of what observed e.g. in \cite{CGX}, where instead players following a Nash equilibrium act later than players following a Pareto optimum. 

The paper is organized as follows. Section 2 presents the continuous time model used to characterize the evolution of the electricity price influenced by the current installed power and presents the procedure for parameter estimation. In Section 3 the model is estimated using real Italian data and the pertinent statistical tests are applied for the validation of the model. Section 4 presents the set up for the singular control problem and its analytical solution, in both the cases with impact and with no impact, for a single producer. Section 5 extends these results to the case when $N$ players can install renewable capacity and derives corresponding Pareto optima  in Subsection 5.1, while Subsections 5.2 and 5.3 are devoted to Nash equilibria and the comparison between the two approaches. Section 6 compares the analytical optimal installation strategy obtained in Section 4 with the real installation strategy applied in Italy. Finally, Section 7 presents our conclusions.

\section{The model}

We start by presenting the model introduced in \cite{KV}, which we here extend to more than one renewable electricity source. 

We assume that the fundamental electricity price $S^{x}(s)$, in absence of increments on the level of renewable installed power, evolves accordingly to an Ornstein-Uhlenbeck (O-U) process

\begin{gather}
\begin{cases}
dS^{x}(s)=  \kappa \left( \zeta   - S^{x}(s) \right)ds + \sigma dW(s) \mbox{ $ s > 0 $ } \\
S^{x}(0) = x
\end{cases},
\label{OU}
\end{gather}

\noindent for some constants $\kappa, \sigma,  x > 0$ and $\zeta \in \mathbb{R}$, where $(W(s))_{s \geq 0}$ is a standard Brownian motion defined on a filtered probability space $\left( \Omega, \mathcal{F}, \mathbb{P}\right)$, more rigorous definition and detailed assumptions will be given in the next section.

We represent the increment on the current installed power level with the sum of increasing processes $Y^{y_{i}}_{i}$, where $y_i$ is the initial installed power and the index $i$ stands for the renewable power source type, which in our case are sun and wind. We relate  $Y^{y_{1}}_1$ with solar energy and $ Y^{y_{2}}_2$ with wind energy. We assume that the increment in the current renewable installed power affects the electricity price by reducing the mean level instantaneously at time $s$ by $ \sum_{i = 1}^{2} \beta^{i} Y^{y_{i}}_i (s) $ for some $\beta^{i} > 0$ \cite{KV}, with $i \in \{ 1, 2\}$. Therefore the spot price $S^{x,I}(s)$ evolves according to  

\begin{gather}
\begin{cases}
dS^{x,I}(s)  = \kappa( \zeta - \sum_{i= 1,2} \beta^{i} Y^{y_{i}}_i (s) - S^{x,I}(s))ds + \sigma dW(s)\mbox{ $s > 0 $ } \\
S^{x,I}(0) = x.
\end{cases}
\label{OUI1}
\end{gather}

\noindent The explicit solution of \eqref{OUI1} between two times $\tau$ and $t$, with $0 \leq \tau < t$ is given by

\begin{eqnarray}
S^{x,I}(t) & = & e^{\kappa (\tau - t)} S^{x,I}(\tau) + \kappa \int_{\tau}^{t} e^{\kappa (s-t)}\left(\zeta - \sum_{i= 1,2} \beta^{i} Y^{y_{i}}_{i}(s) \right) ds + \int_{\tau}^{t} e^{\kappa(s-t)}\sigma dW(s)\\
& = &  e^{\kappa (\tau - t)} S^{x,I}(\tau) +  \zeta (1 - e^{\kappa (\tau - t)} ) - \kappa \int_{\tau}^{t} e^{\kappa(s- t)} \sum_{i= 1,2} \beta^{i} Y^{y_{i}}_{i}(s)ds + \int_{\tau}^{t} e^{\kappa(s-t)}\sigma dW(s).
\label{OUI2}
\end{eqnarray}

\noindent The discrete time version of \eqref{OUI2}, on a time grid $0 = t_0 < t_1 < \ldots $, with constant time step $\Delta t = t_{n+1} - t_{n}$ results in the ARX(1) model 

\begin{eqnarray}
X(t_{n+1})  = 
a + b X(t_{n}) +  \sum_{i = 1,2} u^{i}  Z^{i}(t_{n}) 
+ \delta \epsilon(t_{n}). 
\label{OUD}
\end{eqnarray} 

\noindent where $X(t_{0}), X(t_{1}), X(t_{2}), \ldots $ and $Z^{i}(t_{0}), Z^{i}(t_{1}), Z^{i}(t_{2}), \ldots $ are the observation on the time grid, of process $S^{x,I}$ and $Y^{y_{i}}_{i}$ respectively. The random variables $(\epsilon(t_{n}))_{n = \{0, \ldots , N \}} \sim \mathcal{N}(0,1)$ are iid and the coefficients $a$, $b$, $u^1$, $u^2$ and $\delta$ are related with $\kappa$, $\zeta$, $\beta^1$, $\beta^2$ and $\sigma$ by


\begin{gather}
\begin{cases}
a  =  \zeta (1 - e^{-\kappa \Delta t}) \\
b  =  e^{- \kappa \Delta t} \\
u^{1}  =  - \beta^{1} (1 - e^{-\kappa \Delta t}) \\
u^{2}  =  - \beta^{2} (1 - e^{-\kappa \Delta t}) \\
\delta  =  \frac{\sigma \sqrt{1- e^{-2 \kappa \Delta t}}}{\sqrt{2 \kappa}}
\end{cases}.
\label{parameters}
\end{gather}

\noindent The estimation of the discrete time parameters $a$, $b$, $\delta$ and $u_{i}$, $i = 1,2$ can be obtained from ordinary least squares, which gives maximum likelihood estimators. Then, the continuous time parameters $\kappa$, $\zeta$, $\sigma$ and $\beta^{i}$ with $i= 1,2$ can be estimated by solving Equations \eqref{parameters} \cite{brigo}.

\section{Parameter estimation for Italian zonal prices}

In this section we estimate the parameters of the model in Equation \eqref{OUD} using real Italian data of energy price and current installed power.

\subsection{The dataset}

We have data from six main price zones of Italy, which are North, Central North, Central South, South, Sicily and Sardinia. For every zone we have weekly measurements of average energy price in \euro /MWh, together with photovoltaic and wind energy production in MWh. The time series goes from 07/05/2012 to 25/06/2018, week 19/2012 to 26/2018, corresponding to $N = 321$ observations. The time series of current photovoltaic and wind installed power is instead available with a much lower frequency (i.e. year by year). In order to obtain a time series consistent with the weekly granularity of price and production, we estimate the installed power to be proportional to  the running maximum of the photovoltaic and wind energy production of whole Italy, respectively. Summarizing, we use for estimation of the model in Equation \eqref{OUD}, for every particular zone, the data summarized in Table \ref{TD}.

\begin{table}[H]
\centering
\begin{tabular}{| p{2cm} | c | p{9.5cm} |}
\hline
Variable Type & Nomenclature & Description\\
\hline
Time step observation	& $t_{1}, \ldots, t_N$ 	& Weeks when the quantities are observed, $N = 321$. \\
\hline
Response variable & $X(t_0 ), \ldots, X(t_{N})$ & Electricity price in \euro/MWh  relative to an Italian price zone.\\
\hline
Explanatory variable & $Z^{1}(t_0),\ldots,Z^{1}(t_{N})$ & Current  installed photovoltaic power in MW, estimated as 
$Z^{1}(t_{i}) = \max(E^{1}(t_0),\ldots, E^{1}(t_{i}))$, $i \in \{1, \ldots, N \}$, where $E^{1}(t_i)$ is the sum of the produced energy on the six zones at the observation time $t_i$.\\
\cline{2-3}
 & $Z^{2}(t_0),\ldots,Z^{2}(t_{N})$ & Current  installed wind power in MW, estimated as 
 $Z^{2}(t_{i}) = \max\{ E^{2}(t_0), \ldots, E^{2}(t_{i}) \}$, $i \in \{1, \ldots, N \}$, where $E^{1}(t_i)$ is the sum of the produced energy on the six zones at the observation time $t_i$.\\
\hline
\end{tabular}
\caption{The data used for parameter estimation of Equation \eqref{OUD}.}
\label{TD}
\end{table}

\subsection{Results}

Using ordinary least squares considering the data described above and then setting $\Delta t =  t_{i+1} - t_{i} = \frac{1}{52}$  for all $i = 0,\ldots,320$, we obtain, by Equations \eqref{parameters}, the continuous time parameters for the O-U model with an exogenous impact in the mean reverting term, for every zone. Table \ref{T1} shows the estimation results by zone. 

\begin{table}[H]
\centering
\begin{footnotesize}
\begin{tabular}{| c | l | r | r | r | r | r |  r | }
\hline
Zone &  \multicolumn{6}{c |}{parameters} & Box Pierce test \\
\cline{2-7}
 &  & $\kappa$ & $\zeta$ & $\beta^1 $  & $\beta^2 $ & $\sigma$ & $p$-value\\
\hline 
\multirow{2}{*}{North}& Value & *** 10.6056  & *** 133.0670 & *  0.0148 & 0.0012 & *** 47.7527 &   0.6101 \\
    & s.e. & 2.1437  & 32.2392 &  0.0082 & 0.0031 &  2.3741 & \\ 
\hline
\multirow{2}{*}{Central North}& Value & *** 10.9960 & *** 120.4933  &   0.0112 & 0.0027 & *** 45.5106 & 0.2702 \\
    & s.e. & 2.1599 & 30.1593 & 0.0076 & 0.0029 & 2.1413 & \\
\hline
\multirow{2}{*}{Central South}& Value & *** 13.2276 & *** 100.3647 & 0.0052 & ** 0.0056 & *** 45.4237 & 0.0093
\\
    & s.e. & 2.3958 & 27.3713 & 0.0069 & 0.0026  & 2.05040 & \\
\hline
\multirow{2}{*}{South}& Value & *** 11.4996 & *** 98.5810 & 0.0059 & * 0.0047 & *** 41.5805 & 0.0086
\\
    & s.e. & 2.2004 & 26.9193 & 0.0068 & 0.0026 & 1.7715 & \\
\hline
\multirow{2}{*}{Sicily}& Value & *** 14.1614 & ** 173.0264 & 0.0124 & *** 0.0107 & *** 81.4377 & 0.0132
\\
    & s.e. & 2.5146 & 46.9427 & 0.0120 & 0.0044 & 6.4833 & \\
\hline
\multirow{2}{*}{Sardinia}& Value & *** 18.4580 & *** 94.7809 & 0.0020 & ** 0.0129 & *** 68.2290 & 0.1216 \\
    & s.e. & 2.9547 & 33.1946 & 0.0085 & 0.0031 & 4.2260 & \\
\hline
\end{tabular}
\end{footnotesize}
\caption{Estimated parameters for the Ornstein Uhlenbeck. Significance code: *** $= p < 0.01$, ** $= p < 0.05$, * $= p < 0.1$.}
\label{T1}
\end{table}




In Table \ref{T1}, under each parameter we observe the value of every estimator and its respective standard error. Moreover, for each price zone, we include the results of the Box-Pierce test to check the independence of the residuals. This test rejects the independence hypothesis for $p$-values less than $0.05$. According to the results in Table \ref{T1}, the Central South, South and Sicily zones present correlation in the residuals, therefore the proposed O-U model for electricity price is not the right choice. On the other hand the North, Central North and Sardinia zones have independent residuals implying that the model is able to explain the behavior of the electricity price. Regarding the parameters significance for this latter three zones, only the North and Sardinia zones present price impact: in the North there is only photovoltaic impact while in Sardinia only wind impact. We re-estimate the parameters considering only the zones which pass the Box-Pierce test 
and with only the significant price impact parameters. Table \ref{T1a} summarizes the obtained results.

\begin{table}[H]
\centering
\begin{footnotesize}
\begin{tabular}{| c | l | r | r | r | r | r |  r | }
\hline
Zone &  \multicolumn{6}{c |}{parameters} & Box Pierce tests \\
\cline{2-7}
 &  & $\kappa$ & $\zeta$ & $\beta^1 $  & $\beta^2 $ & $\sigma$ & $p$-value \\
\hline 
\multirow{2}{*}{North}& Value & *** 10.3702 & *** 140.5894 & ** 0.0172 & 0 & *** 47.6586 &  0.6206 \\
    & s.e. &  2.0514 & 26.4732 &  0.0054 & &  2.3747 & \\ 
\hline
\multirow{2}{*}{Central North}& Value & *** 9.2648 & *** 55.6085 & 0 & 0 & *** 65.9346 & 0.2771\\
    & s.e. & 1.9273 & 2.8265 &  & &  4.6367 & \\
\hline
\multirow{2}{*}{Sardinia}& Value & *** 18.5248 & *** 102.4620 & 0 & *** 0.0123  & *** 68.2889 & 0.1296\\
    & s.e. &  2.9510 & 6.6813 &  & 0.0017 & 4.2260 & \\
\hline
\end{tabular}
\end{footnotesize}
\caption{Significant estimated parameters for Ornstein Uhlenbeck . Significance code: $*** = p < 0.01$, $** = p < 0.05$, $* = p < 0.1$}
\label{T1a}
\end{table}


\section{The optimal installation  problem}

In this section we give the general set up and description for the singular control problem of optimally increasing the current installed power in order to maximize the profit of selling the produced energy in the market net of the installation cost. This problem is completely described and solved in \cite{KV} when $\beta > 0$.  However, the case when $\beta = 0$ can be obtained using the same procedure, which we describe in this section. Also we include a brief description and practical results of the case when $\beta > 0$ for completeness of the paper. 

\subsection{General set up and description of the problem}

Let  $(\Omega , \mathcal{F}, (\mathcal{F}_{t})_{t \geq 0} , \mathbb{P})$ be a complete filtered probability space where a one-dimensional Brownian motion $W$ is defined and $(\mathcal{F}_{t})_{t \geq 0}$ is the natural filtration generated by $W$, augmented by the $\mathbb{P}$-null sets.  

 As we have already seen, only one type of energy influences the energy price in each price zone: either photovoltaic or wind, but not both simultaneously, therefore in the sequel we use the model in Equation \eqref{OUI1} with only one single process influencing the mean reverting term of the price dynamics. Therefore, we assume that the spot price $S^{x,I}(s)$ evolves according to  
\begin{gather}
\begin{cases}
dS^{x,I}(s)  = \kappa( \zeta - \beta Y^{y} (s) - S^{x,I}(s))ds + \sigma dW(s)\mbox{ $s > 0 $ } \\
S^{x,I}(0) = x.
\end{cases}
\label{OU}
\end{gather}
where the stochastic process $Y = (Y^{y}(s))_{s \geq 0}$, with initial condition $y \in  [0,\theta]$, represents the current renewable installed power of a company, which can be increased irreversibly by installing more renewable energy generation devices, starting from an initial installed power $y \geq 0$, until a maximum $\theta$. This strategy is described by the control process $I = (I(s))_{s \geq 0}$ and takes values on the set $\mathcal{I}[0,\infty)$ of admissible strategies, defined by

$$\mathcal{I}[0,\infty) \triangleq \{I:[0,\infty)\times \Omega \rightarrow [0, \infty) : \mbox{$I$ is $(\mathcal{F}_{t})_{t \geq 0}$- adapted , $t \rightarrow I(t)$ is increasing, cadlag,}$$
$$\mbox{ with $I(0-) = 0 \leq I(t) \leq \theta- y$, $\forall t \geq 0$} \}.$$

\noindent Hence the process $Y^y $ is written as 

\begin{eqnarray}
Y^{y}(t) = y + I(t).
\end{eqnarray}

As we already said, the aim of the company is to maximize the expected profits from selling the produced energy in the market, net of the total expected cost of installing a generation device, which for an admissible strategy $I$ is described by the following utility functional

\begin{eqnarray}
\mathcal{J}(x,y,I) & = & \mathbb{E} \left[\int_{0}^{\infty}  e^{-\rho \tau} S^{x,I}(\tau)a Y^{y}(\tau)d\tau - \int_{0}^{\infty } c e^{-\rho \tau} dI(\tau) \right],
\label{utility}
\end{eqnarray}

\noindent where $\rho > 0$ is a discount factor, $c$ is the installation cost of $1$ MW of technology, $a > 0$ is the conversion factor of the installed device's rated power to the effective produced power per time unit and $S^{x}(s)$ is the electricity price, with $x$ as initial condition. The objective of the company is to maximize the functional in Equation \eqref{utility} by finding an optimal strategy $\hat{I} \in \mathcal{I}[0,\infty)$ such that 

\begin{eqnarray}
V(x,y) = \mathcal{J}(x,y,\hat{I}) & = & \sup_{ I \in \mathcal{I}[0,\infty)} \mathcal{J}(x,y,I).
\label{vfs}
\end{eqnarray}

\subsection{The optimal solution when $\beta \geq 0$}

To make the paper self contained we present in this section the systematic procedure to construct the optimal solution and characterize the value function \eqref{vfs}. All the results presented here are proved in \cite{KV}. We add some comments on the no-impact case $\beta =0$ which is not explicitly treated in \cite{KV}, but it can be derived as particular case.  

Recall that the electricity price evolves accordingly to the O-U process in Equation \eqref{OUI1}. Notice that, for a non-installation strategy $I(s) \equiv 0$ $\forall s \geq 0$, we have

\begin{eqnarray}
\mathcal{J}(x , y, 0) = \mathbb{E} \left[ \int_{0}^{\infty} e^{- \rho s} S^{x}(s) a y ds\right] =  \frac{a xy}{\rho + \kappa} + \frac{a \zeta \kappa y}{\rho (\rho + \kappa)} - \frac{a \kappa \beta y^{2}}{\rho (\rho + \kappa)} =: R(x , y).
\label{r}
\end{eqnarray}

The possible strategies that the company can follows at time zero are: do not install during a time period $\Delta t$ and earn money selling the energy already installed, or immediately install more power. The first strategy carries one equation which is obtained applying the dynamic programming principle and the second one carries an equation obtained by perturbing the value function \eqref{vfs} in the control. As a result we arrive to a variational inequality that the candidate value function $w$ should satisfy, which is

\begin{eqnarray}
\max \left\{ \mathcal{L}w(x, y) - \rho w(x,y) + a x y, \frac{ \partial w}{\partial y} - c \right\} = 0,
\label{HJB0}
\end{eqnarray}

\noindent with boundary condition $w(x,\theta) = R(x,\theta)$ and the differential operator $\mathcal{L}$ defined as

\begin{eqnarray}
\mathcal{L}^{y}u(x,y) & = & \kappa \left( \left(\zeta - \beta y\right)- x\right)  \frac{\partial u(x,y)}{\partial x} + \frac{\sigma^{2}}{2} \frac{\partial^{2} u(x,y)}{\partial x^{2}} .
\end{eqnarray}

\noindent Equation \eqref{HJB0} defines two regions: a waiting region $\mathbb{W}$ and an installation region $\mathbb{I}$, given by

\begin{eqnarray}
\mathbb{W} = \Bigg\{ (x, y) \in \mathbb{R} \times [0, \theta): \mathcal{L}w(x, y) - \rho w(x,y) + a x y = 0, \frac{ \partial w}{\partial y} - c < 0 \Bigg\},\\
\label{wr}
\mathbb{I} = \Bigg\{(x, y) \in \mathbb{R} \times [0, \theta): \mathcal{L}w(x, y) - \rho w(x,y) + a x y \leq 0, \frac{ \partial w}{\partial y} - c = 0 \Bigg\}.
\label{ir}
\end{eqnarray}

\noindent which define when it is optimal to install more power or not.  

It is proved in \cite[Theorem 3.2]{KV} that the solution of \eqref{HJB0} with linear growth identifies with the value function $V$. 

Additionally, it is proved that these two regions are separated by the strictly increasing function $F:[0,\theta] \rightarrow \mathbb{R}$ \cite[Corollary 4.5]{KV}, called the free boundary. Therefore $\mathbb{W}$ and $\mathbb{I}$ can be written as 

\begin{eqnarray}
\mathbb{W} = \{ (x,y) \in \mathbb{R} \times [0,\theta): x < F(y) \},
\label{wrF}
\end{eqnarray} 
\begin{eqnarray}
\mathbb{I} = \{ (x, y) \in \mathbb{R} \times [0, \theta): x \geq F(y) \}.
\label{irF}
\end{eqnarray}

\noindent Now we can describe the optimal strategy using \eqref{wrF} and \eqref{irF}. When the current electricity price $S^{x}(t)$ is sufficiently low, such that $S^{x}(t) < F(Y^{y}(t))$, then the optimal choice is to not increment the installed power until the electricity price crosses $F(Y^{y}(t))$, passing to the installation region, where the optimal choice is to increase the installed power in order to maintain the pair price-power $(S^x(t),Y^y(t))$ not below of the free boundary. Once $S^{x}(t) \geq F(\theta)$ the optimal choice is restricted to increased immediately the installed power level up to the maximum $\theta$. We explain again this strategy in Section 6 observing the numerical solutions and graphics obtained for the Italian case.


Setting  $\hat{F}(y) = F(y) + \beta y$, the free boundary is characterized by the ordinary differential equation \cite[Proposition 4.4 and Corollary 4.5]{KV}

\begin{gather}
\begin{cases}
\hat{F}^{'}(y) = \beta \times \displaystyle \frac{N(y, \hat{F}(y))}{D(y, \hat{F}(y))}, \mbox{ $y \in [0, \theta)$}\\
\hat{F}(\theta) = \hat{x}.
\end{cases}
\label{ODE}
\end{gather}

\noindent where
 
\begin{eqnarray*}
N(y,z) & = & \left(\psi(z) \psi^{''}(z)- \psi^{'}(z)^{2} \right) \left( \frac{\rho + 2\kappa}{\rho} \psi^{'}(z) + \left((\rho + \kappa)\left(c - \hat{R}(z,y) \right)\psi^{''}(z) + \psi^{'}(z)  \right)  \right),
\end{eqnarray*}

\begin{eqnarray*}
D(y,x) = \psi(x) \left( (\rho + \kappa)(c - \hat{R}(x,y))\left(\psi^{'}(x) \psi^{'''}(x) - \psi^{''}(x)^{2} \right) + \psi(x) \psi^{'''}(x) - \psi^{'}(x) \psi^{''}(x)\right)
\end{eqnarray*}

\noindent and the function $\psi$ is the strictly increasing and positive fundamental solution of the homogeneous equation $\mathcal{L} w(x, y) - \rho w(x,y) = 0$ (see  in \cite[Lemma 4.3]{KF} or in \cite[Lemma A.1]{KV}), given by

\begin{eqnarray} \label{psi}
\psi(x) & = & \frac{1}{\Gamma(\frac{\rho}{\kappa})} \int_{0}^{\infty}t^{\frac{\rho}{\kappa} - 1 } e^{-\frac{t^{2}}{2}-\left(\frac{x - \zeta}{\sigma} \sqrt{2 \kappa}\right)t} dt
\end{eqnarray}

\noindent and

\begin{eqnarray}
\hat{R}(x,y) = \frac{a \zeta \kappa + a \rho x - a \beta(\rho + 2 \kappa)y}{\rho(\rho + \kappa)}.
\label{rhat}
\end{eqnarray}

\noindent On the other hand, the boundary condition $\hat{x}$ in \eqref{ODE} is the unique solution of 

\begin{eqnarray} \label{xhat}
\psi'(x)(c - \hat{R}(x,\theta)) + (\rho + \kappa)^{-1} \psi (x) & = & 0.
\label{Acca}
\end{eqnarray}

\begin{remark}
The solution $\hat{x}$ is such that $\hat{x} \in \left( \bar{c}, \bar{c} + \frac{\psi(\bar{c})}{\psi^{'}(\bar{c})} \right)$, with $\bar{c} = c(\rho + \kappa)- \frac{\zeta \kappa - \beta(\rho + 2 \kappa)\theta}{\rho}$ \cite[Lemma 4.2]{KV}.
\label{R1}
\end{remark}

\subsection{The case $\beta = 0$}

When there is not impact, i.e., $\beta = 0$, we have $\hat{F}(y) \equiv F(y)$, then from \eqref{ODE} every $y \in [0, \theta)$, $F^{'}(y) \equiv 0$, hence the free boundary is a constant with value $F(y) = \hat{x}$, with $\hat{x}$ the same solution of \eqref{Acca}, considering $\beta = 0$ in the function $\hat{R}(x,y)$ defined in \eqref{rhat}. Notice that in this case $\hat{R}$ does not depend on $y$.

The candidate value function is given by

\begin{gather}
w(x ,y) =
\begin{cases}
A(y ) \psi(x ) + R(x , y ) \mbox{   , if } (x , y) \in \mathbb{W} \cup ( \{ \theta \}  \times (- \infty , \hat{x} )) \\
R(x, \theta) - c(\theta - y ) \mbox{   , if } (x , y) \in \mathbb{I} \cup ( \{  \theta \}  \times ( \hat{x} , \infty  )) 
\end{cases},
\label{shjb}
\end{gather} 

\noindent with $R(x , y)$ defined in Equation \eqref{r}, $\psi(x)$ given by Equation \eqref{psi} and $A(y)$ given by

\begin{eqnarray}
A(y) =  \frac{\theta - y }{(\rho + \kappa) \psi^{'}(\hat x)}.
\end{eqnarray}

\noindent The optimal control is written as (see \cite[Theorem 4.8]{KV}
\begin{gather}
\hat{I}(t) = \begin{cases}
0 \mbox{ , } t \in [0, \tau) \\
\theta - y \mbox{ , } t \geq \tau
\end{cases},
\label{oss}
\end{gather}

\noindent with $\tau = \inf \{t \geq 0, X(t) \geq \hat{x} \}$.  


\color{black}

\section{A market with $N$ producers}

As mentioned in the Introduction,  Italy has a liberalized market, thus there is not a single producer which can impact prices by him/herself as is assumed in Section 4. Conversely, prices are impacted by the cumulative installation of all the power producers which are present in the market. 
For this reason, we now consider a market with $N$ producers, indexed by $i = 1,\ldots,N$.
The cumulative irreversible installation strategy of the producer $i$ up to time $s$, denoted by $I_{i}(s)$, is an adapted, nondecreasing, cadlag process, such that $I_{i}(0) = 0$. We assume that the aggregated installation of the $N$ firms is allowed to increase until a total maximum constant power $\theta$, that is,

\begin{eqnarray}
\sum_{i = 1}^{N} \left( y_{i} + I_{i}(s) \right) \leq \theta   \mbox{ $\mathbb{P} -$ a.s., }  s \in [0,\infty),
\end{eqnarray}

\noindent where $y_{i}$ is the initial installed power for the firm $i$ and indicate by $\bar{y} = (y_1 , \ldots , y_N )$ the vector of the initial conditions. We denote by $\mathcal{I}_{N}$ the set of admissible strategies of all the players
$$ \mathcal{I}_{N} \triangleq \{\bar{I}:[0,\infty) \times \Omega \rightarrow [0, \infty)^{N} \mbox{ , non decreasing, left continuous adapted process} $$
$$ \mbox{ with $I_{i}(0) = 0$, $\mathbb{P}$-a.s., $\sum_{i=1}^{N} ( y_{i} + I_{i}(s)) \leq \theta$} \}.$$
and notice that each player is constrained, in its strategy, by the installation strategies of the other players.

\subsection{Pareto optima}

We now consider the cooperative situation of a social planner, where the problem consists into finding a efficient installation strategy $\hat{I} \in \mathcal{I}_{N}$ which maximizes the aggregate expected profit, net of investment cost \cite{ferrari}.  While in many liberalized markets there is not a single being which can {\em impose} a given strategy to all the players, this is equivalent to solving a cooperative game with the maximum possible coalition containing all the players.   

The social planner problem, therefore is expressed as 
\begin{eqnarray}
V_{SP} = \sup_{\bar{I} \in \mathcal{I}_{N}} \mathcal{J}_{SP}(\bar{I}),
\label{vfsoc}
\end{eqnarray}

\noindent where

\begin{eqnarray}
\mathcal{J}_{SP}(\bar{I}) = \sum_{i =1}^{N} \mathcal{J}_{i}(I_{i})
\label{sp}
\end{eqnarray}

\noindent and for $i = 1,2,\ldots ,N$,

\begin{eqnarray}
\mathcal{J}_{i}(x, \bar{y}, \bar{I}) & = & \mathbb{E} \left[\int_{0}^{\infty} e^{-\rho \tau}S^{x,\bar{y},\bar{I}}(s)(\tau)a(y_{i} + I_{i}(\tau)) d\tau - c \int_{0}^{\infty} e^{- \rho \tau} dI_{i}(\tau) \right],
\label{UNC}
\end{eqnarray}

\noindent where $\rho$, $a$ and $c$ are the same defined in \eqref{utility}. The process $S^{x, \bar{y}, \bar{I}}(s)$ is the electricity price affected by the sum of the installations of all the agents which, in analogy with the one-player case, we assume to follow an O-U process with an exogenous mean reverting term, whose dynamics is given by 


\begin{gather}
\begin{cases}
dS^{x, \bar{y}, \bar{I}}(s)  = \kappa( \zeta -  \beta \sum_{i= 1}^{N} \left( y_{i} + I_{i}(s) \right) - S^{x, \bar{y}, \bar{I}}(s))ds + \sigma dW(s)\mbox{  $s > 0 $  ,} \\
S^{x, \bar{y}, \bar{I}}(0) = x.
\end{cases}
\label{OUI1N}
\end{gather}
Call now $\nu(t) = \sum_{i =1}^{N}  I_{i}(t)$ and $ \gamma = \sum_{i =1}^{N}  y_{i}$: then, by substituting on the social planner functional \eqref{sp}, we get
\begin{eqnarray}
\mathcal{J}_{SP}(\bar{I}) & =  &\sum_{i =1}^{N} \mathbb{E} \left[\int_{0}^{\infty} e^{-\rho \tau} S^{x, \bar{y}, \bar{I}}(\tau)a (y_{i} + I_{i}(\tau)) d\tau - c \int_{0}^{\infty}e^{-\rho \tau} dI_{i}(\tau)  \right] \\
& = & \mathbb{E} \left[\int_{0}^{\infty}e^{-\rho \tau} S^{x, \bar{y}, \bar{I}}(\tau)a \left(\sum_{i =1}^{N} y_{i} + \sum_{i =1}^{N}  I_{i}(\tau)\right) d\tau - c \int_{0}^{\infty} e^{-\rho \tau}d \left(\sum_{i =1}^{N} I_{i}(\tau)\right) \right]\\
& = & \mathbb{E} \left[\int_{0}^{\infty} e^{-\rho \tau} S^{x, \bar{y}, \bar{I}}(\tau) a ( \gamma + \nu(\tau)) d\tau - c \int_{0}^{\infty} e^{-\rho \tau} d \nu (\tau) \right].
\label{jspf}
\end{eqnarray}

\noindent Observe that we have the same optimal control problem as in the single company case (Section 4), therefore we can guess that the optimal solution for the social planner will be equal to that for the single company. The aggregate optimal strategy for the $N$ producer of a given region results to be Pareto optimal (see Lemma \ref{L} below). 

\begin{lemma}
If $\hat{I} \in \argmax \mathcal{J}_{SP}(I)$, then $\hat{I}$ is Pareto optimal.
\label{L}
\end{lemma}

\begin{proof}
Suppose $\hat{I} \in \argmax \mathcal{J}_{SP}(\bar{I})$ and assume $\hat{I}$ is not Pareto optimal, then there exist $I^{*}$ such that,

\begin{eqnarray}
\mathcal{J}_{i}(I^{*}_{i}) \geq \mathcal{J}_{i}(\hat{I}_{i}) \mbox{ , $\forall$ $i \in \{ 1,\ldots, N  \}$}
\end{eqnarray}

\noindent where at least one inequality is strict. Then,

\begin{eqnarray}
\sum_{i= 1}^{N} \mathcal{J}_{i}(I^{*}_{i}) > \sum_{i= 1}^{N} \mathcal{J}_{i}(\hat{I}_{i}),
\end{eqnarray}

\noindent contradicting the fact that $\hat{I}$ is maximizing.
\end{proof}

As already said in the Introduction, with this approach it is not possible to distinguish the single optimal installations of each producer, as we can only characterize the cumulative installation $\nu(t) = \sum_{i =1}^{N}  I_{i}(t)$, while the single components $I_{i}(t)$ remain to be determined. However, our declared aim is about the effective cumulative installation strategy which was carried out in Italy during the time period covered by the dataset. Thus, in the next section we compare this with the optimal one which we obtained theoretically. 

In the next subsections, instead, we compare these Pareto optima, obtained by assuming that players would cooperate to achieve the maximum cumulative payoff, with Nash equilibria, which instead assume that players compete actively to individually maximize their own payoff. 

\subsection{Nash equilibria in the case $N = 2$}

The Pareto optima found previously for the social planner problem assume a collaboration between players: nevertheless, it could be also possible to have competition in the market between the players, therefore it makes sense to study the non cooperative case and search for Nash equilibria. In particular we solve the case with two players and we compare both results.

The formulation for the competitive game with two players states as follows: the electricity price evolves according to $\eqref{OUI1N}$ and every player aims to maximize its own utility \eqref{UNC}. 
In this case, in analogy with \cite{GTX}, we will look for a subset of the admissible strategies $\mathcal{I}_{2}$, which we describe next.

\begin{definition}(Markovian strategy and admissible control set) A strategy $I(t) \in \mathcal{I}$ is called Markovian if 
$I(t) = I(S(t), Y^{1}(t-), Y^{2}(t-))$ for all $t \geq 0$, where 
$I$ is a deterministic function of the states immediately before time $t$. \color{black} We define the admissible set of Markovian strategies as follows
$$\mathcal{I}_{2}^{M} := \left\{ I_1 , I_2 \in \mathcal{I}_{2}\ |\ (I_1 , I_2) \mbox{ are Markovian strategies} \right\} \subset \mathcal{I}_{2}.$$
\end{definition}

\begin{definition}(Markovian Nash equilibrium)  We say that $ \bar{I}^{*} = (I_{1}^{*} , I_{2}^{*}) \in \mathcal{I}_2^{M} $ is a Markovian Nash equilibrium if and only if for every \color{black} $x \in \mathbb{R}$ and $\bar{y} = (y_1 ,y_2) \in [0, \theta] \times [0, \theta]$, we have 

\begin{eqnarray*}
|\mathcal{J}_{i}(x,\bar{y}, \bar{I}^{*}) | < \infty  \mbox{  , $i = 1,2$} 
\end{eqnarray*}

\noindent and


\begin{gather}
\begin{cases}
\mathcal{J}_{1}(x, \bar{y}, I_{1}^{*} , I_{2}^{*}) \geq \mathcal{J}_{1}(x, \bar{y}, I_{1} , I_{2}^{*}) \mbox{ for any }  I_{1}, \mbox{ such that } (I_{1}, I_{2}^{*}) \in \mathcal{I}_{2} ,\\
\mathcal{J}_{2}(x, \bar{y}, I_{1}^{*} , I_{2}^{*}) \geq \mathcal{J}_{2}(x, \bar{y}, I_{1}^{*} , I_{2}) \mbox{ for any }  I_{2}, \mbox{ such that } (I_{1}^{*}, I_{2}) \in \mathcal{I}_{2} .
\end{cases}
\end{gather}

The value function corresponding to the Nash equilibrium for each player $i$ is defined as 

\begin{eqnarray}
V_{i}(x, \bar{y}) : = \mathcal{J}(x, \bar{y} , \bar{I}^{*}).
\end{eqnarray}

\end{definition}

For each player we also define the waiting and installation regions, for Markovian Nash equilibria, defined as follows \cite{GTX}.

\begin{definition}(Installation and waiting regions) The installation region of player $i$ is defined as the set of points $\mathbb{I}_i \subseteq \mathbb{R} \times [0,\theta]^2$ such that $dI^*_i(t) \neq 0$ if and only if $(X(t), Y_1(t-),Y_2(t-)) \in \mathbb{I}_i$,  
%
%
and its waiting region as $\mathbb{W}_i = \mathbb{I}_{i}^{c}$.
\end{definition}


We derive the Hamilton-Jacobi-Bellman equation following this heuristic argument: by the Markovian structure it is enough to observe the case at time $t = 0$. For agent $i$, it can decide to do not increase the current level of installed power and also player $j$, i.e., the strategy is $\bar{I} = \bar{I}^{0} \equiv (0,0)$ and both continue optimally.
 In this case, the control problem reduces to the single player case and we have 

\begin{eqnarray*}
V_{i}(x,\bar{y} ) \geq \mathbb{E}\left[ \int_{0}^{\Delta t} e^{-\rho s} a S^{x, \bar{y}, \bar{I}^{0} }(s) y_i ds + e^{- \rho \Delta t} V_{i}(S^{x, \bar{y}, \bar{I}^{0} }(\Delta t), \bar{y} )   \right] \mbox{ , } 
\end{eqnarray*}

\noindent leading to

\begin{eqnarray}
\mathcal{L}^{\bar{y} }V_{i}(x,\bar{y} )- \rho V_{i}(x,\bar{y} ) + axy_i \leq 0 \mbox{ , } 
\end{eqnarray}

\noindent with $\mathcal{L}^{\bar{y}}$ the differential operator defined by

\begin{eqnarray}
\mathcal{L}^{\bar{y} } u(x,\bar{y}  ) = \sigma \frac{\partial^{2} u(x,\bar{y} )}{\partial x^2 } + \kappa \left( \zeta - x - \beta
\sum_{i=1}^{2}y_i \right)\frac{\partial u(x,\bar{y})}{\partial x}.
\end{eqnarray}

\noindent Conversely, player $i$ can decide to increase its level by $\epsilon$ while player $j$ does not increase its level, then both continue optimally
, which is associated with

\begin{eqnarray}
V_{i}(x, \bar{y} ) \geq V_{i}(x, \bar{y} + e_{i} \epsilon ) - c \epsilon,
\end{eqnarray} 

\noindent where $e_i$ is the canonical vector in the direction $i$. Dividing by $\epsilon$ and $\epsilon \downarrow 0$, we get

\begin{eqnarray}
0 \geq  \frac{\partial V_{i}(x,\bar{y} )}{\partial y_i}  - c.
\end{eqnarray}

\noindent Let us assume instead that player $i$ decides to not increase its level while player $j$ increases its level. By definition of Nash equilibrium, player $i$ is not expected to suffer a loss, therefore


\begin{eqnarray}
V_i(x,\bar{y}) \geq V_i (x, \bar{y} + e_{j}\epsilon), 
\end{eqnarray}
\noindent where $e_{j}$ is the canonical vector in the direction $j$. Dividing the above expression by $\epsilon$ and letting $\epsilon\downarrow 0$, we obtain
\begin{eqnarray}
\frac{\partial V_{i}(x,\bar{y})}{\partial y_{j}}  \leq 0.
\end{eqnarray}

\noindent Finally, if instead both players decide to increase their level by $\epsilon$ and continue optimally, 
this is associated with

\begin{eqnarray}
V_{i}(x, \bar{y} ) \geq V_{i}(x, \bar{y} + (1,1) \epsilon ) - c \epsilon,
\end{eqnarray} 

\noindent dividing by $\epsilon$ and $\epsilon \downarrow 0$, we get

\begin{eqnarray}
0 \geq  \frac{\partial V_{i}(x,\bar{y} )}{\partial y_i} +  \frac{\partial V_{i}(x,\bar{y})}{\partial y_j} - c.
\end{eqnarray}

\noindent The above arguments suggest that the value function of player $i = 1,2$, $V_i (x, \bar{y} )$ should be identified with  a \color{black} solution of the following variational inequality

\begin{gather}
\begin{cases}
\max \left\{ \mathcal{L}^{\bar{y} }w_i(x,\bar{y} )- \rho w_i(x,\bar{y} ) + axy_i , \frac{ \partial w_i (x,\bar{y} )}{\partial y_i}  - c  \right\} = 0, & \mbox{ $(x,\bar{y}  ) \in \mathbb{W}_{j}$ }\\
\max \left\{ \frac{\partial w_{i}}{\partial y_{j}} , \sum_{k = 1}^{2} \frac{ \partial w_i (x,\bar{y} )}{\partial y_k}  - c  \right\} = 0, & \mbox{ $(x,\bar{y}  ) \in \mathbb{I}_{j}$ }
\end{cases}
\label{vi2n}
\end{gather}


\noindent with $i \neq j$ and with the boundary condition $w_i(x,\bar{y} ) = R_i(x,\bar{y} )$ whenever $\sum_{i=1}^{2}y_i = \theta$, where

\begin{eqnarray*}
R_i (x,\bar{y} ) & := & \mathcal{J}_i (x,\bar{y}, \bar{I}^{0}) = 
\mathbb{E} \left[ \int_{0}^{\infty} e^{-\rho s} aS^{x,\bar{y}, \bar{I}^{0}}(s)y_i ds \right] \\
& = & \frac{axy_i}{\rho + \kappa} + \frac{a\zeta \kappa y_i}{\rho (\rho + \kappa)} - \frac{a\kappa \beta  y_i  \sum_{i=1}^{2}y_i}{\rho (\rho + \kappa)}.
\end{eqnarray*}

\begin{remark}
We point out that, in stochastic singular games, the usual framework is that a player can act only when the other ones are idle, i.e. $\mathbb{I}_i \cap \mathbb{I}_j = \emptyset$ for all $i \neq j$, see e.g. \cite{CGX,DeAFer,GTX,Xin}. Here instead the variational inequality \eqref{vi2n}, and the argument before it, takes explicitly into consideration the possibility that both players acts (i.e. install) simultaneously. 
This possibility will be confirmed in Section 5.3, where the presented Nash equilibrium will even have both players acting and waiting simultaneously, i.e. $\mathbb{I}_i = \mathbb{I}_j$.
\end{remark}

Now we establish a verification theorem for the value function.

\begin{theorem}(Verification theorem) For any $i = 1,2$, suppose $\bar{I}^{*} \in \mathcal{I}_2^{M}$, the corresponding $w^{i}(\cdot) = \mathcal{J}(\cdot; \bar{I}^{*})$ satisfies the following: 

\begin{itemize}

\item[(i)] { $w_{i} \in C^0( \mathbb{R} \times [0, \theta]^{2} ) \cap C^{2,1,1}(\mathbb{W}_j)$, with $j \neq i$;}

\item[(ii)] $w_i$ satisfies the growth condition
\begin{eqnarray}
| w_{i}(x,y_1 , y_2 )| \leq K(1 + |x|); 
\end{eqnarray}
\item[(iii)] $w_i$ satisfies Equation \eqref{vi2n}, 
with $i \neq j$, with the boundary condition $w_i(x,\bar{y} ) = R_i(x,\bar{y} )$, whenever $\sum_{i=1}^{2}y_i = \theta$;

\end{itemize}

\noindent then $\bar{I}^{*}$ is a Nash equilibrium with value function $w_i$ for each player $i = 1,2$.
\end{theorem}











{ 
\begin{remark}
Differently from the one-player case, where the value function is required to be of class $C^2$ (or at least smooth enough for the Ito formula to be applied) in the whole domain, here each candidate value function $w_i$ is required to be smooth only in the continuation region ${\mathbb W}_j$ of the other player as, under a Nash equilibrium, the state will not exit from there. In fact, player $j$ will not deviate from $I^*_j$, thus making ${\mathbb I}_j$ inaccessible: for this reason, player $i$ will be allowed to change its controls only in ${\mathbb W}_j$. 
This is analogous with other results on singular control games based on variational inequalities, see e.g. \cite{DeAFer,GTX,Xin}
\end{remark}
}

\begin{proof} Let $(x,\bar{y} ) \in \mathbb{R} \times [0,\theta)^{2} $ be given and fixed, and $(I_{i}, I^{*}_{j}) = \bar{I} \in \mathcal{I}_2^{M} $. Denote by $\Delta I^{i}(s) = I_{i}(s) - I_{i}(s-)$ and $I_{i}^{c}$ the continuous part of the strategy $I$. Define $\tau_{R,N} : = \tau_{R} \wedge N$, where $\tau_{R} = \inf\{ s > 0 : S^{x, \bar{y}} \notin (-R,R) \}$.  Applying the Ito formula to   $e^{-\rho \tau_{R,N}}w_{i}(S^{x,\bar{y} }(\tau_{R,N}) , Y_{i}(\tau_{R,N}) , Y_{j}^{*}(\tau_{R,N}))$ \color{black} we have

\begin{eqnarray}
& & e^{-\rho \tau_{R,N}}w_{i}(S^{x,\bar{y}, \bar{I} }(\tau_{R,N}) , Y_{i}(\tau_{R,N}) , Y_{j}^{*}(\tau_{R,N})) -  w_{i}(x,y_i ,y_{j}) = \\
& & \int_{0}^{\tau_{R,N}} \left(  -\rho e^{- \rho s}w_{i}(S^{x,\bar{y} , \bar{I}}(s) , Y_{i}(s) , Y_{j}^{*}(s)) + e^{-\rho s} \mathcal{L}^{\bar{y}  }w_{i}(S^{x,\bar{y} , \bar{I}}(s) , Y_{i}(s) , Y_{j}^{*}(s)) \right)ds\\
& & + \int_{0}^{\tau_{R,N}} \sigma \frac{\partial w_i (S^{x,\bar{y} , \bar{I} }(s) , Y_{i}(s) , Y_{j}^{*}(s))}{\partial x} dW(s) \nonumber \\ 
& & + \int_{0}^{\tau_{R,N}} e^{- \rho s} \frac{\partial w_{i}(S^{x,\bar{y}, \bar{I}}(s) , Y_{i}(s) , Y_{j}^{*}(s))}{\partial y_i}dI_{i}^{c}(s) + \int_{0}^{\tau_{R,N}} e^{- \rho s} \frac{\partial w_{i}(S^{x,\bar{y} , \bar{I} }(s) , Y_{i}(s) , Y_{j}^{*}(s))}{\partial y_j}dI_{j}^{*c}(s)\nonumber \\
& & + \sum_{0 \leq s \leq \tau_{R,N}} e^{-\rho s} \left[ w_{i}(S^{x,\bar{y}, \bar{I} }(s) , Y_{i}(s) , Y_{j}^{*}(s)) - w_{i}(S^{x,\bar{y}, \bar{I}  }(s) , Y_{i}(s-) , Y_{j}^{*}(s-)) \right].
\label{pvt}
\end{eqnarray}

\noindent Set $\Delta Y_{k}(s) = Y_{k}(s) - Y_{k}(s-)$, $k = 1,2$ and notice that

\begin{eqnarray*}
& &w_{i}(S^{x,\bar{y} , \bar{I}}(s) , Y_{i}(s) , Y_{j}^{*}(s)) - w_{i}(S^{x,\bar{y} , \bar{I} }(s) , Y_{i}(s-) , Y_{j}^{*}(s-)) \\ 
& & =   \int_{0}^{1} \left[ \frac{\partial w_{i}(S^{x,\bar{y}, \bar{I} }(u) , Y_{i}(u) , Y_{j}^{*}(u))}{\partial y_i } \Delta Y_{i}(u)  +  \frac{\partial w_{i}(S^{x,\bar{y} , \bar{I}}(u) , Y_{i}(u) , Y_{j}^{*}(u))}{\partial y_j } \Delta Y_{j}(u)  \right] du.
\end{eqnarray*}


\noindent Considering the above expression, taking expectation in \eqref{pvt}, observing that the process
 $$ \left( \int_{0}^{\tau} \sigma \frac{\partial w_i (S^{x,\bar{y}, \bar{I}  }(s) , Y_{i}(s) , Y_{j}^{*}(s))}{\partial x} dW(s) \right)_{\tau \geq 0} $$ is a martingale and using assumptions $(ii)$, we have

\begin{eqnarray*}
& & w_{i}(x,y_i ,y_{j}) + K \mathbb{E}\left[ e^{-\rho \tau_{R,N}}\left( 1 + |S^{x,\bar{y}, \bar{I} }(\tau)|\right)\right] \geq\\
&  & = \mathbb{E} \left[ \int_{0}^{\tau_{R,N}} \left(  \rho e^{- \rho s}w_{i}(S^{x,\bar{y} , \bar{I} }(s) , Y_{i}(s) , Y_{j}^{*}(s))  - e^{-\rho s} \mathcal{L}^{\bar{y}  }w_{i}(S^{x,\bar{y} , \bar{I} }(s) , Y_{i}(s) , Y_{j}^{*}(s)) \right)ds \right. \\
& & - \int_{0}^{\tau_{R,N}} e^{- \rho s} \frac{\partial w_{i}(S^{x,\bar{y} , \bar{I}}(s) , Y_{i}(s) , Y_{j}^{*}(s))}{\partial y_i}dI_{i}^{c}(s) - \int_{0}^{\tau_{R,N}} e^{- \rho s} \frac{\partial w_{i}(S^{x,\bar{y}  }(s) , Y_{i}(s) , Y_{j}^{*}(s))}{\partial y_j}dI_{j}^{*c}(s)\\
& & \left. - \sum_{0 \leq s \leq \tau_{R,N}} e^{-\rho s} \int_{0}^{1} \left[ \frac{\partial w_{i}(S^{x,\bar{y} , \bar{I}}(u) , Y_{i}(u) , Y_{j}^{*}(u))}{\partial y_i } \Delta Y_{i}(u)  +  \frac{\partial w_{i}(S^{x,\bar{y}, \bar{I} }(u) , Y_{i}(u) , Y_{j}^{*}(u))}{\partial y_j } \Delta Y_{j}(u)  \right] du\right].\\
\end{eqnarray*}

\noindent Using the variational equation of assumption (iii), we get

\begin{eqnarray*}
& & w_{i}(x,y_i ,y_{j}) + K \mathbb{E}\left[ e^{-\rho \tau_{R,N}}\left( 1 + |S^{x,\bar{y} , \bar{I}}(\tau)|\right)\right] \\
&  & \geq\mathbb{E} \left[ \int_{0}^{\tau_{R,N}}  e^{- \rho s}  aS^{x,\bar{y}, \bar{I} }(s)Y_{i}(s) ds \right. \\
& & - \int_{0}^{\tau_{R,N}} e^{- \rho s} \frac{\partial w_{i}(S^{x,\bar{y}, \bar{I} }(s) , Y_{i}(s) , Y_{j}^{*}(s))}{\partial y_i}dI_{i}^{c}(s) - \int_{0}^{\tau_{R,N}} e^{- \rho s} \frac{\partial w_{i}(S^{x,\bar{y} , \bar{I} }(s) , Y_{i}(s) , Y_{j}^{*}(s))}{\partial y_j}dI_{j}^{*c}(s)\\
& & \left. - \sum_{0 \leq s \leq \tau_{R,N}} e^{-\rho s} \int_{0}^{1} \left[ \frac{\partial w_{i}(S^{x,\bar{y} , \bar{I}}(u) , Y_{i}(u) , Y_{j}^{*}(u))}{\partial y_i } \Delta Y_{i}(u)  +  \frac{\partial w_{i}(S^{x,\bar{y} , \bar{I}}(u) , Y_{i}(u) , Y_{j}^{*}(u))}{\partial y_j } \Delta Y_{j}(u)  \right] du\right]\\
& & \geq \mathbb{E} \left[ \int_{0}^{\tau_{R,N}}  e^{- \rho s}  aS^{x,\bar{y}, \bar{I} }(s)Y_{i}(s) ds  - \int_{0}^{\tau_{R,N}} e^{- \rho s} \frac{\partial w_{i}(S^{x,\bar{y} , \bar{I}}(s) , Y_{i}(s) , Y_{j}^{*}(s))}{\partial y_i}dI_{i}^{c}(s) \right.\\
& &\left. - \sum_{0 \leq s \leq \tau_{R,N}} e^{-\rho s} \int_{0}^{1} \left[ \frac{\partial w_{i}(S^{x,\bar{y}, \bar{I} }(u) , Y_{i}(u) , Y_{j}^{*}(u))}{\partial y_i } \Delta Y_{i}(u) \right] du\right]\\
& & \geq \mathbb{E} \left[ \int_{0}^{\tau_{R,N}}  e^{- \rho s}  aS^{x,\bar{y} , \bar{I}}(s)Y_{i}(s) ds - c \int_{0}^{\tau_{R,N}} e^{- \rho s}dI_{i}(s) \right]. 
\end{eqnarray*}

\noindent We can apply the dominated convergence theorem in the last expression since (see proof \cite[Theorem 3.2]{KV} for the computations of the following estimates)




\begin{eqnarray*}
\mathbb{E} \left[ \int_{0}^{\tau}  e^{- \rho s}  aS^{x,\bar{y}, \bar{I} }(s)Y_{i}(s) ds - c \int_{0}^{\tau} e^{- \rho s}dI_{i}(s) \right] \leq \theta \int_{0}^{\infty} e^{-\rho s} \left(  |S^{x,\bar{y},, \bar{I}^{0}}(s)| + \kappa \beta \theta s \right)ds + c\theta
\end{eqnarray*}

\noindent and 

\begin{eqnarray}
\mathbb{E}\left[ e^{-\rho \tau_{R,N}}\left( 1 + |S^{x,\bar{y}, \bar{I} }(\tau_{R,N})| \right) \right] \leq C_1 \mathbb{E}\left[ e^{- \rho \tau_{R,N}} (1 + \tau_{R,N}) \right] + C_3 \mathbb{E}\left[ e^{-\rho \tau_{R,N}}\right]^{1/2} (1 + x^{2}).
\end{eqnarray}

\noindent Letting $N \uparrow \infty$ and $R \uparrow \infty$, we get

\begin{eqnarray*}
\mathcal{J}(x,\bar{y}, I_{i}, I_{j}^{*}) \leq w_{i}(x, \bar{y}),
\end{eqnarray*}

\noindent for all $I_{i}$ such that $(I_{i}, I^{*}_{j}) \in \mathcal{I}_{2}^{M}$, therefore $\bar{I}^{*}$ is a Markovian Nash equilibrium.

\end{proof}







\subsection{The case $\beta = 0$: comparison between Pareto optimum and Nash equilibrium}

While a complete characterization of Nash equilibria in the general case appears to be technically very challenging and is beyond the scope of this article, here we analyze the case without price impact, i.e. with $\beta = 0$. Inspired by the one-player optimal control and by the $N$-players Pareto optima, we search for a Nash equilibrium $\bar I^*$ where the players, which have initial installation equal to $(Y_1(0), Y_2(0)) = (y_1,y_2)$, wait until the price surpasses a boundary $x^*$ to be determined, and then they make together a cumulative installation which completely saturates the total capacity $\theta$. Following the arguments of the previous subsection, we assume that they share equally this additional installation. 

More in detail, we define
\begin{equation} \label{taster}
\tau^* := \inf\{ t \geq 0\ |\ S(t) \geq x^* \}
\end{equation}
and describe the Nash equilibrium $\bar I^*$ as 
\begin{equation} \label{Istar}
\bar I^*(t) := \frac12 (\theta - y_1 - y_2) (1,1) \mathbf{1}_{t \geq \tau^*} 
\end{equation}
Obviously, in this case $\mathbb{I}_1 = \mathbb{I}_2 = (x^*,+ \infty) \times [0,\theta]^2$. 
For each player $i = 1,2$, the value function which corresponds to this strategy can be computed as follows: 
\begin{eqnarray*}
w_{i}(x,\bar{y}) & = & \mathbb{E}\left[ \int_{0}^{\tau^*} ae^{-\rho s} S^{x, \bar{y}}(s) y_i ds 
+ \int_{\tau^*}^{\infty} ae^{-\rho s} S^{x,\bar{y}}(s) \left( y_i + \frac{\theta - y_i - y_j }{2} \right)ds - \frac{ce^{-\rho \tau^* }(\theta - y_i -y_j)}{2}  \right]\\
& = & R_{i}(x,\bar{y}) + \frac{1}{2}\mathbb{E} \left[ e^{-\rho \tau^*} \int_0^{\infty} a e^{-\rho s} S^{x,\bar{y}}(\tau^* + s) \left( \theta - y_i - y_j\right)ds -ce^{-\rho \tau^* }(\theta - y_i -y_j) \right]\\
& = & R_{i}(x,\bar{y}) + \frac{1}{2}\mathbb{E} \left[ e^{-\rho \tau^*} R_i(S^{x,\bar{y}}(\tau^*),\theta - y_i - y_j,y_j) - c e^{-\rho \tau^* }(\theta - y_i -y_j) \right]
\end{eqnarray*}
where in the last equality we use the strong Markov property for the process $S$. Now, if $x < x^*$, then $\tau^* > 0$ and $\mathbb{E}[ e^{-\rho \tau^*}] = \frac{\psi(x)}{\psi(x^*)}$, with $\psi$ as in Equation \eqref{psi} \cite[Chapter 7.2]{BorSal}, and 
\begin{eqnarray*}
w_{i}(x,\bar{y}) & = & R_{i}(x,\bar{y}) + \frac{1}{2}\mathbb{E} \left[ e^{-\rho \tau^* }  \right]\left(R_i(x^* , \theta - y_i - y_j,y_j) - c(\theta - y_i -y_j) \right)\\
& = & R_{i}(x,\bar{y}) + \frac{\psi (x)}{2 \psi(x^* )} \left(R_i(x^* , \theta - y_i - y_j,y_j) - c(\theta - y_i -y_j) \right),
\label{wb01}
\end{eqnarray*}
Instead, when $x \geq x^*$, then $\tau^* \equiv 0$ and 
\begin{eqnarray*}
w_{i}(x,\bar{y}) 
& = & R_{i}(x,\bar{y}) + \frac{1}{2} \left( R_i(x,\theta - y_i -y_j,y_j) - c(\theta - y_i -y_j) \right) 
\label{wb02}
\end{eqnarray*}
\noindent Therefore, for a given level $x^*$, the value function for the strategy \eqref{Istar} is given by
\begin{gather}
w_{i}(x,\bar{y}) = \begin{cases}
R_{i}(x,\bar{y}) + \frac{\psi(x)}{2 \psi(x^*)} \left( R(x^*,\theta- y_i - y_j) - c(\theta- y_i - y_j ) \right),  & x < x^*\\
R_{i}(x,\bar{y}) + \frac{1}{2 } \left( R(x,\theta- y_i - y_j) - c(\theta- y_i - y_j ) \right),  & x^* \geq x
\end{cases}
\label{wb0}
\end{gather}

If we let $x^* := \hat x$ as the solution of Equation \eqref{xhat}, then the corresponding strategy is one of the Pareto optima found in Lemma 5.1. However, if we plug the candidate value functions of Equation \eqref{wb0} into the variational inequality \eqref{vi2n}, it turns out that this choice does {\em not}  give a Nash equilibrium. Instead, a Nash equilibrium is achieved when we let
\begin{equation} \label{cbar}
x^* := \bar c = \frac{c (\rho + \kappa)}{a} - \frac{\xi \kappa}{\rho} 
\end{equation}

\begin{proposition} If $x^* = \bar c$ defined in Equation \eqref{cbar}, then the strategy \eqref{Istar} is a Nash equilibrium and the value function for player $i = 1,2$ is given by \eqref{wb0}.
\end{proposition}

\begin{proof} The function $w_i \in C^0( \mathbb{R} \times [0, \theta]^{2} ) \cap C^{2,1,1}(\mathbb{W}_j)$ by direct computations and it has linear growth by \cite[Theorem 3.2, Lemma 4.6]{KV}. Let us check that it satisfies the variational inequality \eqref{vi2n}. First of all, the boundary condition $w_i (x,y_i , y_j) = R(x, y_i + y_j)$ whenever $y_i + y_j = \theta$ is fulfilled by direct computations.

Then, for player $i = 1,2$, in order to verify the variational inequality \eqref{vi2n}, we distinguish two cases.

\textit{Case 1}: For player $i$, $(x,\bar{y}) \in \mathbb{W}_j $. In this case we also have $(x,\bar{y}) \in \mathbb{W}_{i}$ and $x < x^* = \bar c$. We expect $w_i $ satisfies $\mathcal{L}^{\bar{y}}w_i - \rho w_i + axy_i = 0$: in fact, 

\begin{eqnarray*}
\mathcal{L}^{\bar{y}}w_i(x,\bar{y}) - \rho w_i(x,\bar{y}) + axy_i & = & \mathcal{L}^{\bar{y}}(R_{i}(x,\bar{y}) + \frac{\psi (x)}{2 \psi(x^*)}(R(x^*, y_i + y_j ) - c(\theta -y_i - y_j))) \\
& &- \rho (R_{i}(x,\bar{y}) + \frac{\psi (x)}{2 \psi(x^*)}(R(x^*, y_i + y_j ) - c(\theta -y_i - y_j))) + axy_i\\
& = & \left( \mathcal{L}^{\bar{y}} - \rho \right)R_{i}(x,\bar{y}) + axy_i \\
& = & \frac{a\kappa(\zeta  -x)y_i }{\rho + \kappa} - \frac{\rho axy_i}{\rho + \kappa}  - \frac{ a \zeta \kappa y_i}{\rho + \kappa} + a x y_i = 0.
\end{eqnarray*}


\noindent Also, when $x < \bar c$ we should have $\frac{\partial w_i}{\partial y_i} -c \leq 0$, and in fact  
\begin{eqnarray*}
\frac{\partial w_{i}(x,y_i ,y_j ) }{\partial y_i} - c & = & \frac{a}{\rho + \kappa} \left(x + \frac{\xi \kappa}{\rho} \right) + \frac{\psi (x)}{2 \psi(x^*)} \left( - \frac{a}{\rho + \kappa} \left(x^* + \frac{\xi \kappa}{\rho} \right) + c \right) - c \\
& \leq & \left( \frac{a}{\rho + \kappa} \left(x + \frac{\xi \kappa}{\rho} \right) - c \right) \left(1 - \frac{\psi (x)}{2 \psi(x^*)} \right) = \\
& = & \frac{a}{\rho + \kappa} (x - \bar c) \left(1 - \frac{\psi (x)}{2 \psi(x^*)} \right) < 0
\end{eqnarray*}
as $\psi$ is strictly increasing. 





\textit{Case 2:} For player $i$, when $(x,\bar y) \in \mathbb{I}_j$ then also $(x,\bar y) \in \mathbb{I}_i$. 
We expect $ \frac{\partial w_i (x,y_i , y_j)}{\partial y_i} + \frac{\partial w_i (x,y_i , y_j)}{\partial y_j} - c = 0$: in fact, 


\begin{eqnarray*}
\sum_{k = i,j} \frac{\partial w_i (x,y_i , y_j)}{\partial y_k}  - c 
& = & \frac{\partial (R_i (x,y_i , y_j) - R(x,\theta - y_i - y_j)/2)}{\partial y_i} + \frac{c}{2} + \\
& & + \frac{\partial (R_i (x,y_i , y_j) - R(x,\theta - y_i - y_j)/2) }{\partial y_j} + \frac{c}{2} - c = \\
& = & \frac{ax}{(\rho +  \kappa)} + \frac{a \zeta \kappa}{\rho (\rho + \kappa)} - \frac{ax}{(\rho + \kappa)} - \frac{a \zeta \kappa}{ \rho (\rho + \kappa)} = 0.
\end{eqnarray*}

\noindent On the other hand, when $x \geq \bar c$ we also expect that $\frac{\partial w_i (x,y_i , y_j)}{\partial y_j} \leq 0$: in fact, 


\begin{eqnarray*}
\frac{\partial w_i (x,y_i , y_j)}{\partial y_j} & = & \frac{1}{2} \left( - \frac{ax}{\rho + \kappa} - \frac{a \zeta \kappa}{ \rho (\rho + \kappa)} + c \right) = 
\frac{a}{2 (\rho + \kappa)} (\bar c - x) \leq 0
\end{eqnarray*}




\end{proof}

\begin{remark} \label{nashbeforepareto}
Since, after Remark 4.1, we have $\bar c < \hat x$, this means that the search for a Nash equilibrium induces the agents to perform an earlier installation with respect to the cooperative behavior of the Pareto optimum seen in the previous section. This phenomenon is the converse of the one observed in \cite{CGX}, where instead the Nash equilibrium's action regions are contained in the Pareto optima's ones, i.e. agents wait more under the Nash equilibrium than under the Pareto optimum. By continuity, we expect a similar behavior also for the case $\beta > 0$, at least for low values of $\beta$: in other words, also in the case when price impact is present, competitive Nash equilibria will induce players to install earlier than when they would install under a cooperative Pareto optimum. We reserve to investigate this topic furtherly in future research.
\end{remark}




\color{black}
\section{Numerical verification}

In this section we solve numerically Equation \eqref{ODE}, using the parameters' values estimated in Section 3 for the North, Central North and Sardinia zones. 

Following the spirit of Section 5.1, we treat the pool of producers in each zone as a coalition maximizing the cumulative payoff and thus realizing a Pareto optimum. We choose not to report results about Nash equilibria, as the analysis in Sections 5.2 and 5.3 contains only partial results; however, after Remark \ref{nashbeforepareto}, we expect that a free boundary relative to a Nash equilibrium would always be located on the left of the Pareto optimum, which instead we explicitly describe below.

Recall from Table \ref{T1a} that the price impact  in the North zone is due to photovoltaic power production, while in Sardinia is due to wind power production. Both are cases when the parameter impact is $\beta > 0$, which we describe in Section 4.1.2. On the other hand, Central North has not price impact from power production (at least from these two renewable sources), so here we are in the case $\beta = 0$ described in Section 4.1.1.


The parameters $c$ and $a$ presented in \eqref{utility} should be selected according to the type of renewable energy which has an impact on the corresponding price zone. In the case of photovoltaic power we consider a yearly average of the installation cost of $1$ MW of the prices available in the market, see e.g. \cite{PV}. On the other hand, for the wind power installation cost we consider the invested money on an offshore wind park that is being installed in Sardinia \cite{W}. In both cases we adjust for 
the presence of government incentives for renewable energy installation (usually under the form of tax benefits), 
therefore we consider around a $40\%$ and a $60 \%$ of the real investment cost $c$ of photovoltaic and wind power, respectively, for our numerical simulation. The parameter $a$ is the effective power produced during a representative year: as we consider a yearly scale for simulation, the parameter $a$ will convert our weekly data of produced power into yearly effective produced power. Additionally, the $a$ value depends on the type of produced power. This information is available e.g. in \cite[Chapter 4]{EFV}.
The parameter $\rho$ is the discount factor for the electricity market and is the same in the three cases: no impact, photovoltaic and wind power impact. The parameter $\theta$ is the effective power that can be produced considering the real installed power of the respective type of energy.
In the case of the estimated parameters $\kappa$, $\zeta$, $\beta$ and $\sigma$, we choose a value from the $95\%$ confidence interval, based on better heuristic numerical performance simulation criteria.

We summarize in Table \ref{TS} the parameters considered for the numerical simulations.

\begin{table}[H]
\centering
\begin{tabular}{| c | c | c | c | c | c | c | c | c |}
\hline
Zone &  \multicolumn{8}{c |}{Parameters' values}\\
\cline{2-9}
& $\kappa$ & $\zeta$ &  $\beta$ & $\sigma$ & $c$ & $a$ & $\theta$ & $\rho$\\
\hline
North & 6.7 & 124.7 & 0.0091 &  47.7 & 290000 & 1400 &  6500 & 0.1\\
\hline
Central North & 5.6029 & 50.2381 & 0 & 58.9796 & 290000 & 1400 & 6500 & 0.1\\
\hline
Sardinia & 13.213 & 115.1565 & 0.0091 & 68.2889 & 1944400 & 7508 & 5700 & 0.1\\
\hline
\end{tabular}
\caption{Parameter values used for  the North, Central North and Sardinia zones.}
\label{TS}
\end{table}


For the Central North case, we consider the cost of photovoltaic installation, because it is the main renewable energy produced in this zone.

\subsection{North}

We solved the ordinary differential equation \eqref{ODE} using the data in Table \ref{TS} for the North, using the backward  Euler scheme with step $h = 0.5$ and initial condition 
$
\hat{F}(\theta) = 976.4 \mbox{ \euro/MWh},
$
\noindent which was obtained by solving Equation \eqref{Acca} with the bisection method considering as initial points the extremes of the interval on Remark \ref{R1}. The graph of the solution for the free boundary $F(y) = x$ is presented in Figure \ref{F1a}, with a detail on realized power prices in Figure \ref{F1b}.

\begin{figure}[H]
\begin{subfigure}[t]{0.45\textwidth}
\includegraphics[scale=0.4]{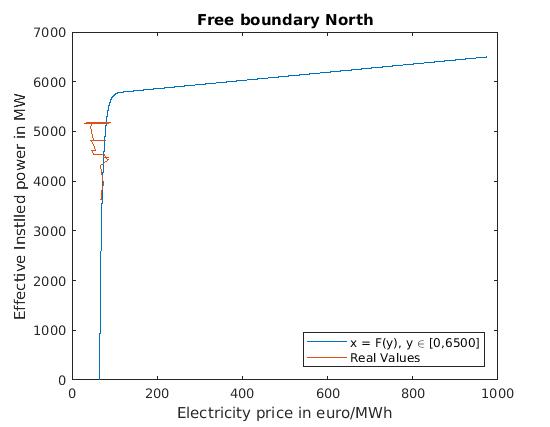}
\caption{Simulated free boundary and real data for the North}
\label{F1a}
\end{subfigure}%
\hfill
\begin{subfigure}[t]{0.45\textwidth}
\includegraphics[scale=0.4]{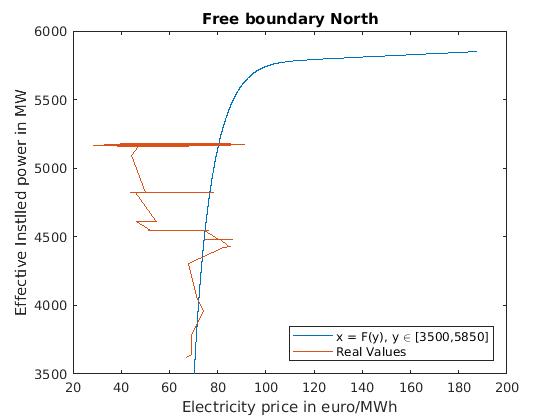}
\caption{Detail of free boundary and real data for the North}
\label{F1b}
\end{subfigure}
\end{figure}

In Figure \ref{F1a}, the point at zero installation level corresponds to $F(0) = 64.9$ \euro/MWh. The red irregular line corresponds to the realized trajectory $t \to (X(t),Y(t))$, i.e. to the values of electricity price vs effective photovoltaic installed power in the North: from it we can see that, at the beginning of the observation period (2012), the installed power\footnote{recall that $Y$ is really just an estimation of the installed power, which is officially given with yearly granularity; moreover, $Y$ is expressed in units of rated power, i.e. in production equivalent to a power plant always producing the power $Y$} was already around 3600 MW. 
Instead, the blue smooth line corresponds to the computed free boundary $F(y) = x$, which expresses the optimal installation strategy in the following sense: when the electricity price $S^x(t)$ is lower than $F(Y(t))$, i.e. when we are in the waiting region (see \eqref{wrF}), no installation should be done and it is necessary to wait until the price $S^x(t)$ crosses $F(Y(t))$ to optimally increase the installed power level. When the electricity price $S^x(t)$  is between $F(0)$ and $ F(\theta)$, enough power should be installed to move the pair price-installation in the up-direction 
until reaching the free boundary $F$. In the extreme case when $S^x(t) \geq F(\theta)$ the energy producer should install instantaneously the maximum allowed power $\theta$. In the detailed Figure \ref{F1b} we can observe the strategy followed in the North zone: the installation level from $3500$ MW until $4500$ MW was approximately optimal, in the sense that the pair price-installed power was around the free boundary $F$, with possibly some missed gain opportunities when, between 4300 and 4500 MW, the price was deep into the installation region; nevertheless, the rise in renewable installation from $4500$ MW to $4800$ MW was at the end done with a power price which resulted lower than what should be the optimal one. At around $4800$ MW, there was an optimal no installation procedure until the price entered again the installation region: 
again, the consequent installation strategy was executed with some delay, resulting in a non-optimal strategy. At the end of the installation (around 5200 MW), we can see that the pair price-installed power moved again deep into the installation region: we should then expect an increment in installation.

\subsection{Central North}

In this case we do not have price impact, hence the constant free boundary $F(y) \equiv \bar{x}$ was obtained solving Equation \eqref{xhat}. 
As before, we used the bisection method considering as  initial points the extremes of the interval described in Remark \ref{R1}. The obtained value is $
%
F(y) = \bar x = 29.3205 \mbox{ \euro/MWh}.
$
\begin{figure}[H]
\begin{subfigure}[t]{0.45\textwidth}
\includegraphics[scale=0.4]{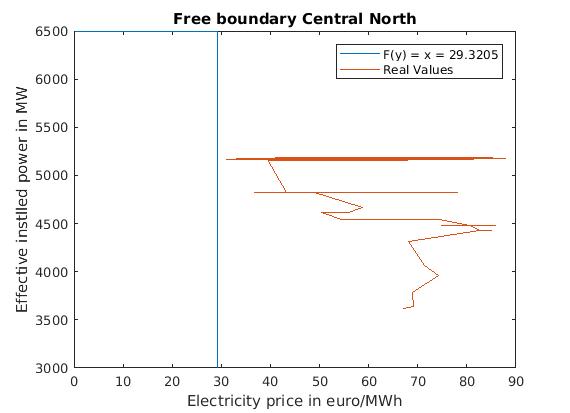}
\caption{Simulated free boundary and real data for Central North}
\label{F2a}
\end{subfigure}
\begin{subfigure}[t]{0.45\textwidth}
\includegraphics[scale=0.4]{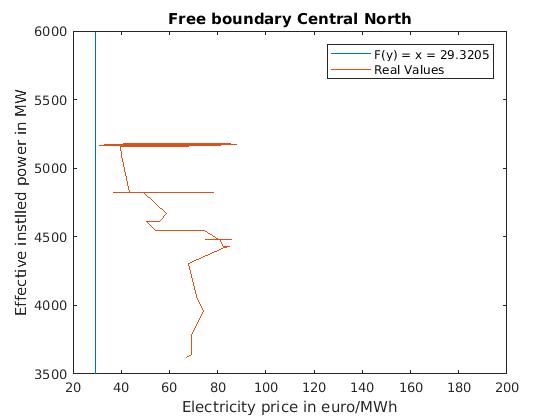}
\caption{Detail of free boundary and real data for Central North}
\label{F2b}
\end{subfigure}
\end{figure}

In Figure \ref{F2a} the vertical blue line corresponds to the constant free boundary $ \bar{x} = 29.3205$ \euro/MWh, while the red irregular line with the realized values of price-installation action that was put in place in the Central North zone. In this case, the optimal strategy is described as follows: for electricity prices less than $\bar{x}$, no increments on the installation level should be done. Conversely, when the electricity price is grater or equal to $\bar{x}$ the producer should increment the installation level up to the maximum level allowed for photovoltaic power (here we posed $\theta = 6500$ MW). As we can clearly see on Figure \ref{F2a}, the electricity price has always been greater than $\bar{x}$ in the observation period; however,  the increments on the installation level was not high enough to arrive to the maximum level $\theta = 6500$ MW, therefore the performed installation was not optimal.

\subsection{Sardinia}

As in the North case, we solved the differential equation \eqref{ODE} using the data in Table \ref{TS} for Sardinia, using the backward  Euler scheme with step $h = 0.2$ and initial condition $
%
\hat{F}(\theta) = 1453.3 \mbox{ \euro/MWh},
$
\noindent which was obtained by solving Equation \eqref{Acca} using the bisection method and considering as initial points the extremes of the interval in Remark \ref{R1}. The graph of the solution for the free boundary $F(y) = x$ is presented in Figure \ref{F3a}.

\begin{figure}[H]
\begin{subfigure}[b]{0.5\textwidth}
\includegraphics[scale=0.4]{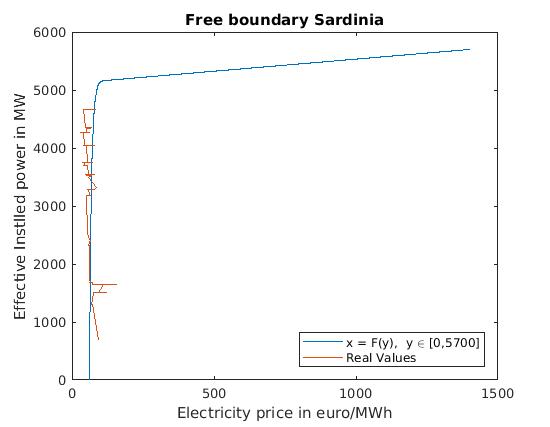}
\caption{Simulated free boundary and real data for Sardinia}
\label{F3a}
\end{subfigure}
\begin{subfigure}[b]{0.5\textwidth}
\includegraphics[scale=0.4]{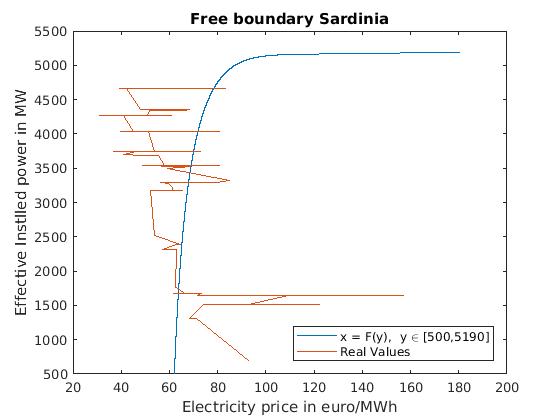}
\caption{Detail free boundary and real data for Sardinia}
\label{F3b}
\end{subfigure}
\end{figure}

In Figure \ref{F3a} the point at zero installation level corresponds to $F(0) = 61.5199$ \euro/MWh . The red irregular line corresponds with the realized values of electricity price vs effective wind installed power in Sardinia, from which we can see that the installed wind power at the beginning of the observation period was already around 600 MW.  The blue smooth line corresponds to the simulated free boundary $F(y) = x$, which expresses the optimal installation strategy as was already explained for the North case. In the detailed Figure \ref{F3b} we can observe the strategy followed in the Sardinia zone: until the level $1600$ MW the power price was very deeply into the installation region, but the installation increments were not high enough to be optimal. Optimality came between the levels $1600$ MW and $2400$ MW, where the performed strategy was to effectively maintain the pair price-installed power around the free boundary $F$. However, the subsequent increments were not optimal, in the sense that the installed power was often increased in periods where the electricity price was too low, and in other situations the power price entered deeply in the installation region without the installed capacity following that trend, or rather doing it with some delay. 

\subsection{Discussion}

We must start by saying that we did not expect optimality in the installation strategy. In fact, firstly this strategy has been carried out by very diverse market operators, including hundreds of thousands of private citizens mounting photovoltaic panels on the roof of their houses, thus not necessarily by rational agents which solved the procedure shown in Sections 4 and 5. Moreover, we must also say that renewable power plants like photovoltaic panels or wind turbines often meet irrational resistances by municipalities, especially when performed at an industrial level: more in detail, photovoltaic farms are perceived to "steal land" from agriculture (see e.g. \cite{DGLS}), while high wind turbines are generically perceived as "ugly" (together with many other perceived drawbacks, see the exaustive monography \cite{ChaCri} on this). 

Despite all these possible adverse effects we saw that, in the North and Sardinia price zones, part of the realized trajectory of power price and installed capacity was very near to the optimal free boundary, while in other periods the installation was put in place in moments when power price was not the optimal one --- possibly, the installation was planned when the power price was high and deep into the installation region (time periods like this have been described both in the North as in Sardinia, see Sections 6.1 and 6.3) but the installation was delayed by adverse effects like e.g. the ones described above. Summarizing, in these two regions the final installation level resulting at the end of the observation period (2018) seems consistent with the price levels reached during the period. 

It is instead difficult to reach such a conclusion in the Central North region: in fact, in that case the realized trajectory of power price and installed capacity was always deeply into the installation region, as the power price was always above the constant free boundary $F(y) = \bar x$ which resulted in this case: the optimal strategy should then have been to install immediately the maximum possible capacity. We did observe a rise in installed renewable power during the period, which was obviously not optimal in the execution time (which spanned several years), given the peculiar nature of the free boundary. However, in analogy to what already said for the North and Sardinia price zones, it is possible that the performed installation, which at the end took place during the observation period,  has been planned in advance but delayed by the same adverse effects cited above. 

\section{Conclusions}

We apply to real modeling and simulation the model presented in \cite{KV}, which assumes that the electricity price evolves accordingly to an O-U process and that it is affected by renewable power installation on the mean-reverting term. The original model considers one big company that influences the electricity price with its activities. To be more realistic, we also study the case when $N$ producers have an impact on electricity price by their aggregate installation. To solve this $N$-player 
game, we use a "social planner" approach as in \cite{ferrari} and maximize the aggregate utility of the $N$ producers: this approach produces Pareto optima, and brings the problem back to the one-producer case.  We also present an analysis which shows that, if we instead search for noncooperative Nash equilibria, we would obtain strategies where producer install earlier than when they would under a Pareto optimum. 

Using real data from the six main Italian price zones we found that, under an O-U model with an exogenous influence on the mean reverting term, there exists significant price impact of renewable power production in the North and Sardinia zones. Also we found that for the Central North price there is not renewable production impact on power price, which is well described by an O-U model without exogenous term. 

Once we solve numerically the ordinary differential equation for the free boundary or trigger frontier, which describes when it is optimal to increment the installed power, we compare it with the real installation strategy that was put in place in the North, Central North and Sardinia zones. We found that for the North the installation was optimal until the $4500$ MW level, while in Sardinia the installation was optimal between $1600$ MW and $2400$ MW level. On the other hand, the capacity expansion in Central North was executed but not in an optimal way, and the increment on the installation level should possibly be higher than what it was. We also present a discussion on this, stating some possible reasons why the installation has not been fully optimal.

\medskip

\noindent \textbf{Acknowledgments.} T. Vargiolu acknowledges financial support from the research grant BIRD 172407-2017 "New perspectives in stochastic methods for finance and energy markets" and BIRD 190200/19 "Term
structure dynamics in interest rates and energy markets: modelling and numerics", both of the
University of Padova. We wish to thank Giorgio Ferrari, Markus Fischer, Katarzyna Maciejowska and Rafa\l\ Weron for useful comments and suggestions.


\begin{thebibliography}{99}

\bibitem{Zervos} Al Motairi, H., Zervos, M.\ (2017). Irreversible Capital Accumulation with Economic Impact. \textsl{Appl.\ Math.\ Optim.} \textbf{75(3)} 525--551.

\bibitem{Becherer2} Becherer, D., Bilarev, T., Frentrup, P.\ (2017). Optimal Asset Liquidation with Multiplicative Transient Price Impact. \textsl{Appl.\ Math.\ Optim.} {\bf 78 (3)}, 643--676

\bibitem{Becherer} Becherer, D., Bilarev, T., Frentrup, P.\ (2018). Optimal Liquidation under Stochastic Liquidity. \textsl{Finance Stoch.}\ \textbf{22(1)} 39--68.

\bibitem{Benth} Benth, F.E., Benth, J.S., Koekebakker, S.: \textit{Stochastic Modeling of Electricity and Related Markets}. Advanced Series on Statistical Science \& Applied Probability: Volume 11. World Scientific (2008)

\bibitem{Bertola} Bertola, G. (1998). Irreversible Investment. \textsl{Res.\ Econ.}~\textbf{52(1)}, 3--37. 

\bibitem{BorSal} Borodin, W. H., Salminen, P. (2002). {\em Handbook of Brownian motion-facts and formulae (2nd ed.)}. Birkh\"auser.

\bibitem{Borovkova} Borovkova, S., Schmeck, M.D.\ (2017). Electricity Price Modeling with Stochastic Time Change. \textsl{Energy\ Econ.}~\textbf{63}, 51--65. 

\bibitem{BPPB} Bosco, B., Parisio, L., Pelagatti, M., Baldi, F.\ (2010). Long-run Relations in European Electricity Prices. \textsl{J.\ Appl.\ Econom.} \textbf{25}, 805--832. 

\bibitem{brigo}
  Brigo, D., Dalessandro, A., Neugebauer, M., Triki, F.
  (2008).
  A Stochastic Processes Toolkit for Risk Management.
  Available at {\tt https://arxiv.org/abs/0812.4210}.

\bibitem{CarFig} Cartea, \'A., Figueroa, M.G.\ (2005). Pricing in Electricity Markets: A Mean Reverting Jump Diffusion Model with Seasonality. \textsl{Appl.\ Math.\ Finance} \textbf{12(4)}, 313--335.

\bibitem{CFSV} Cartea, \'A., Flora, M., Slavov, G., Vargiolu, T.\ (2019). Optimal cross-border electricity trading. Available at \verb|https://papers.ssrn.com/sol3/papers.cfm?abstract_id=3506915|

\bibitem{CarJai} Cartea, \'A., Jaimungal, S., Qin, Z.\ (2019). Speculative trading of electricity contracts in interconnected locations. \textsl{Energy\ Econ.}\textbf{79} 3-20.

\bibitem{ChaCri} Chapman, S., Crichton, F. (2017). \textit{Wind turbine syndrome --- a communicated disease}. Public and Social Policy Series, Sydney University Press. 

\bibitem{ferrari}
 Chiarolla, M., Ferrari, G., Riedel, F.
 (2013).
 Generalized Kuhn-Tucker Conditions for N-Firm Stochastic Irreversible Investment under Limited Resources. \textit{SIAM Journal on Control and Optimization} {\bf 51 (5)}, 3863--3885
 
\bibitem{CGX} Cont, R., Guo, X., Xu, R. (2020). Pareto Optima for a Class of Singular Control Games. Available at {\tt https://hal.archives-ouvertes.fr/hal-03049246}

\bibitem{DeAFer} De Angelis, T.,  Ferrari, G. (2018). Stochastic nonzero-sum games: a new connection between singular control and optimal stopping. \textit{Advances in Applied Probability} {\bf 50 (2)}, 347 -- 372 
 
\bibitem{DGLS} Dias, L., Gouveia, J.P., Lourenço, P., Seixas, J. (2019). Interplay between the potential of photovoltaic systems and agricultural land use. \textit{Land Use Policy} {\bf 81}, 725--735
 
\bibitem{Dixit} Dixit, A.K., Pindyck, R.S.\ (1994). \textit{Investment under Uncertainty}. Princeton University Press. Princeton.

\bibitem{EFV}
  Edoli, E., Fiorenzani, S., Vargiolu, T. 
  (2016) 
  \textit{Optimization Methods For Gas And Power Markets --- Theory And Cases}. 
  Palgrave Macmillan.   


\bibitem{KF}
  Ferrari, G., Koch, T.
  (2019). An Optimal Extraction Problem with Price Impact.
  \textit{Applied Mathematics and Optimization}.
 Available at  {\tt https://doi.org/10.1007/s00245-019-09615-9}

\bibitem{FonVarZor} Fontini, F., Vargiolu, T., Zormpas, D. (2020). Investing in electricity production under a reliability options scheme. 
\textit{Journal of Economic Dynamics and Control} 
104004. Available at {\tt https://doi.org/10.1016/j.jedc.2020.104004}

\bibitem{German} Geman, H., Roncoroni, A.\ (2006). Understanding the fine Structure of Electricity Prices. \textit{J.\ Bus. }\textbf{79(3)} 1225-1261.

\bibitem{GPP}  Gianfreda, A., Parisio, L., Pelagatti, M. \ (2016). Revisiting long-run Relations in Power Markets with high RES penetration. \textsl{Energy Policy}\ \textbf{94}, 432-445.

\bibitem{GTX} Guo, X., Tang, W., Xu, R. (2020). A class of stochastic games and moving free boundary problems. 
Available at {\tt https://arxiv.org/abs/1809.03459v4}.

\bibitem{Xin} Guo, X., Xu, R. (2019). Stochastic Games for Fuel Followers Problem: N vs MFG. \textsl{SIAM\ J.\ Control\ Optim.} \textbf{57(1)}, 659-692.

\bibitem{guo} Guo, X., Zervos, M.\ (2015). Optimal Execution with Multiplicative Price Impact. \textsl{SIAM\ J.\ Finance\ Math.} \textbf{6(1)}, 281-306.

\bibitem{KV}
  Koch, T., Vargiolu, T.
  (2019).
  Optimal installation of solar panels under permanent price impact. \textsl{SIAM Journal of Control and Optimization}, accepted. 
  Available at {\tt https://arxiv.org/abs/1911.04223}. 

\bibitem{W} Media Duemila, \textit{The wind of energy is rising} (2020). Available at {\tt https://www.eni.com/en-IT/technologies/wind-energy-rising.html}

\bibitem{NYTimes} NY TIMES, December 25, 2017, https://www.nytimes.com/2017/12/25/business/energy-environment/germany-electricity-negative-prices.html.

\bibitem{RVG} Rowi\'nska, P.A., Veraart, A., Gruet, P.\ (2018). A Multifactor Approach to Modelling the Impact of Wind Energy on Electricity Spot Prices. Available at SSRN: {\tt https://ssrn.com/abstract=3110554 or http://dx.doi.org/10.2139/ssrn.3110554}

\bibitem{PV} Schirru, L. (2021). Scopri il costo degli impianti fotovoltaici nel 2021. Available at {\tt https://www.mrkilowatt.it/impianti-fotovoltaici/impianto-fotovoltaico-costo-e-prezzi/} (in Italian). 

\bibitem{Weron} Weron, R., Bierbrauer, M., Tr\"uck, S.\ (2004). Modeling Electricity Prices: Jump Diffusion and
Regime Switching. \textsl{Ph.\ A.} \textbf{336(1-2)}, 39-48.

\end{thebibliography}
\end{document}